\documentclass{amsart}
\usepackage{amsmath,amssymb,amsfonts,amstext,amsthm}
\usepackage{url}
\usepackage{fancyhdr}
\usepackage{parskip}
\usepackage{graphicx}
\usepackage{hyperref}
\usepackage{tikz}
\usepackage[all,cmtip]{xy} 
\usepackage{multicol}

\input xy
\xyoption{all}
\usepackage[all]{xy}

\newcommand{\nwc}{\newcommand}
\newcommand*{\Scale}[2][4]{\scalebox{#1}{$#2$}}%

\nwc{\aaa}{\mathcal{F}}
\nwc{\aap}{\mathcal{F}_{P}}
\nwc{\al}{\alpha}

\nwc{\C}{\mathbb{C}}
\nwc{\cb}{\overline{C}}
\nwc{\ccc}{\mathfrak{c}}
\nwc{\ch}{\widehat{C}}
\nwc{\cin}{\textbf{(v)}}
\nwc{\cl}{C'}
\nwc{\cod}{{\rm cod}}
\nwc{\cp}{\mathcal{C}_{P}}
\nwc{\cpll}{\mathfrak{c}_{P'}}
\nwc{\ct}{\widetilde{C}}

\nwc{\dd}{\mathcal{L}}
\nwc{\ddd}{\mathfrak{d}}
\nwc{\ddl}{\mathcal{L}'}
\nwc{\dlp}{\delta_{P}}
\nwc{\doi}{\textbf{(ii)}}

\nwc{\enq}{$$}

\nwc{\fl}{\flushleft}
\nwc{\fff}{\mathcal{F}}
\nwc{\ffp}{\mathcal{F}_{P}}
\nwc{\ffq}{\mathcal{F}_{Q}}
\nwc{\ffl}{\mathcal{F}'}

\nwc{\G}{\mathcal{G}}
\nwc{\Ga}{\Gamma}
\nwc{\gtl}{\widetilde{g}}
\nwc{\gon}{{\rm gon}}

\nwc{\hra}{\hookrightarrow}
\nwc{\hua}{h^{1}(C,\aaa )}

\nwc{\kk}{{\bf k}}

\nwc{\llb}{\mathcal{L}}

\nwc{\mb}{\mathbb}
\nwc{\mc}{\mathcal}
\nwc{\mm}{\mathfrak{m}}
\nwc{\mmp}{\mathfrak{m}_{P}}
\nwc{\mpd}{\mathfrak{m}_{P}^{2}}

\nwc{\nn}{\mathbb{N}}

\nwc{\ob}{\overline{\mathcal{O}}}
\nwc{\obr}{\mathcal{O}^{*}}
\nwc{\obp}{\overline{\mathcal{O}}_P}
\nwc{\och}{\mathcal{O}_{\hat{C}}}
\nwc{\oh}{\hat{\mathcal{O}}}
\nwc{\ohp}{\hat{\mathcal{O}}_{P}}
\nwc{\ol}{\mathcal{O}'}
\nwc{\oma}{\Omega (\mathfrak{a})}
\nwc{\omo}{\Omega (\mathcal{O})}
\nwc{\oo}{\mathcal{O}}
\nwc{\op}{\mathcal{O}_P}
\nwc{\opc}{\mathcal{O}_{P,C}}
\nwc{\oph}{\hat{\mathcal{O}}_{P}}
\nwc{\opl}{\mathcal{O}_{P}'}
\nwc{\oplc}{\mathcal{O}_{P,C}'}
\nwc{\opll}{\mathcal{O}_{P'}}
\nwc{\opt}{\tilde{\mathcal{O}}_{P}}
\nwc{\optt}{{\mathcal{O}}_{\tilde{P}}}
\nwc{\oq}{\mathcal{O}_{Q}}
\nwc{\oqt}{\tilde{\mathcal{O}}_{Q}}
\nwc{\ot}{\widetilde{\mathcal{O}}}
\nwc{\overop}{\bar{\oo}_{P}}

\nwc{\pb}{\overline{P}}
\nwc{\pbb}{P^{*}}
\nwc{\pbi}{\overline{P_{i}}}
\nwc{\pbr}{\overline{P_{r}}}
\nwc{\pgmd}{\mathbb{P}^{g+2}}
\nwc{\pgmu}{\mathbb{P}^{g+1}}
\nwc{\ph}{\hat{P}}
\nwc{\pp}{\mathbb{P}}
\nwc{\prv}{\noindent\textbook{Proof}:}
\nwc{\pt}{\widetilde{P}}
\nwc{\ptl}{\tilde{P}}
\nwc{\pum}{\mathbb{P}^{1}}

\nwc{\qh}{\hat{Q}}
\nwc{\qtl}{\tilde{Q}}
\nwc{\qua}{\textbf{(iv)}}

\nwc{\ra}{\rightarrow}
\nwc{\rh}{\hat{R}}

\nwc{\sei}{\textbf{(vi)}}
\nwc{\sep}{\beq\ast\ \ast\ \ast\enq}
\nwc{\sig}{\sigma}
\nwc{\Sig}{\Sigma}
\nwc{\ssp}{S_{P}}
\nwc{\sss}{{\rm S}}
\nwc{\sys}{\mathcal{L}}

\nwc{\tre}{\textbf{(iii)}}

\nwc{\um}{\textbf{(i)}}

\nwc{\vpb}{v_{\overline{P}}}
\nwc{\vtxp}{\widetilde{V}_{x,P}}
\nwc{\vxp}{V_{x,P}}
\nwc{\vv}{\mathcal{W}}
\nwc{\vvp}{\mathcal{W}_{P}}
\nwc{\val}{\match{V}}

\nwc{\wh}{\hat{\omega}}
\nwc{\whp}{\hat{\omega}_{P}}
\nwc{\woch}{\omega\cdot\mathcal{O}_{\hat{C}}}
\nwc{\woh}{\omega\cdot\hat{\mathcal{O}}}
\nwc{\ww}{\omega}
\nwc{\wwb}{\omega^{*}}
\nwc{\wwct}{\omega _{\widetilde{C}}}
\nwc{\wwh}{\widehat{\omega}}
\nwc{\wwhp}{\widehat{\omega}_P}
\nwc{\wwp}{\omega _{P}}
\nwc{\wwt}{\widetilde{\omega}}
\nwc{\wwtp}{\widetilde{\omega}_P}

\nwc{\zz}{\mathbb{Z}}

\newtheorem{coro}{Corollary}[section]

\newtheorem{conj}[coro]{Conjecture}
\newtheorem{dfn}[coro]{Definition}

\newtheorem{lemma}[coro]{Lemma}

\newtheorem{prop}[coro]{Proposition}

\newtheorem{rem}[coro]{Remark}

\newtheorem{thm}[coro]{Theorem}
\newtheorem{ex}[coro]{Example}
\let \fl=\flushleft

\let \sub=\subset
\let \be=\beta
\let \al=\alpha
\let \pr=\prime
\let \la=\lambda

\begin{document}

\title{Certified 
Severi dimensions for hyperelliptic and supersymmetric cusps}

\author{Ethan Cotterill}
\address{IMECC, Unicamp,
Rua S\'ergio Buarque de Holanda, 651,
13083-859, Campinas SP, Brazil}
\email{cotterill.ethan@gmail.com}

\author{Vin\'icius Lima}
\address{Departamento de Matem\'atica, ICEx, UFMG,
Av. Ant\^onio Carlos 6627,
30123-970 Belo Horizonte MG, Brazil}
\email{viniciuslaralima@gmail.com}

\author{Renato Vidal Martins}
\address{Departamento de Matem\'atica, ICEx, UFMG,
Av. Ant\^onio Carlos 6627,
30123-970 Belo Horizonte MG, Brazil}
\email{renato@mat.ufmg.br}

\author{Alexandre Reis}
\address{Departamento de Matem\'atica, ICEx, UFMG,
Av. Ant\^onio Carlos 6627,
30123-970 Belo Horizonte MG, Brazil}
\email{alexandre.silva.reis@gmail.com}



\begin{abstract}
In a previous paper, the first three authors formulated a precise conjecture about the dimension of the {\it generalized Severi variety} $M^n_{d,g; {\rm S}, {\bf k}}$ of degree-$d$ holomorphic maps $\mb{P}^1 \ra \mb{P}^n$ whose images' singularities are singleton cusps with value semigroups ${\rm S}$ and ramification profiles ${\bf k}$. In this paper, we prove that an adjusted form of the conjecture holds for generic profiles ${\bf k}$ associated with two distinguished (infinite) classes of semigroups ${\rm S}$.
\end{abstract}

\maketitle
\tableofcontents

\section{Introduction}

The geometry of complex algebraic curves $C$ in a given ambient variety $X$ depends on the dimension of $X$, and on the positivity properties of (the canonical bundles of) $C$ and $X$.
Joe Harris' proof \cite{H1} of the irreducibility of the Severi variety $M^2_{d,g}$ of plane curves of fixed degree $d$ and genus $g$ implies that every curve indexed by a point of $M^2_{d,g}$ lies in the closure of the sublocus of $M^2_{d,g}$ that parameterizes {\it $g$-nodal} rational curves; and analogous irreducibility results are known for some (though not all) other ambient toric surfaces $X$.

\medskip
However, this phenomenon fails when one replaces $M^2_{d,g}$ by $M^n_{d,g}$, the ``Severi variety" of degree-$d$ morphisms $\mb{P}^1 \ra \mb{P}^n$ of degree $d$ and arithmetic genus $g$. Indeed, as we saw in \cite{CLM}, it is easy to construct examples of Severi varieties $M^n_{d,g}$ with components of strictly-larger dimension than that of the $g$-nodal locus for every $n \geq 3$. Our construction was predicated on producing large substrata 
$\mc{V}_{\bf k}$ associated with judiciously chosen semigroups ${\rm S}$ and ramification profiles $\bf{k}$. We also conjectured a general combinatorial formula for the (co)dimension of $\mc{V}_{\bf k}$ as a function of ${\rm S}$ and ${\bf k}$.

\medskip
In the present paper we prove an adjusted version of our Severi dimensionality conjecture for two interesting classes of examples, namely when ${\rm S}$ is a hyperelliptic (resp., {\it supersymmetric}) semigroup and ${\bf k}$ consists of the first $n$ consecutive even numbers (resp., a minimal generating set for ${\rm S}$). Our choice of examples is motivated by our reformulation of the dimensionality conjecture in terms of the so-called {\it Betti elements} of the monoid generated by the components of ${\bf k}$ and the minimal generators of ${\rm S}$. Roughly speaking, the hyperelliptic strata we study are associated with a maximal number of Betti elements; the supersymmetric strata, on the other hand, are associated with singleton Betti sets. To prove the dimensionality conjecture in each of these two cases, we implement an algorithm that allows us to check, order by order, whether or not a given positive integer arises as the valuation of an element of the local algebra of a hyperelliptic or supersymmetric cusp with prescribed ramification profile. An upshot of our analysis is that the corresponding Severi varieties are always unirational.

\medskip
In section~\ref{certifying_hyperelliptic_conditions} we prove the following result about Severi varieties of unicuspidal rational curves with hyperelliptic cusps. 

\begin{thm}(= Theorem~\ref{theo2.1} and Corollary~\ref{hyperelliptic_severi_unirationality})
Given positive integers $d$, $g$, and $n$ for which $n \leq 2g \leq d$, the subvariety $\mathcal{V} \subset M^n_{d,g}$ parameterizing rational curves with a unique singularity of hyperelliptic cuspidal type is of codimension
\[{\rm cod}(\mathcal{V},M^n_d) = (n-1)g\]
and the generic hyperelliptic Severi stratum $\mc{V}_{(2,4,\dots,2n)}$ is unirational.
\end{thm}

In section~\ref{Severi_dimensionality_conjectures}, we revisit and adjust the dimensionality conjecture \cite[Conj. 3.6]{CLM} for Severi varieties of unicuspidal rational curves of type $({\rm S},{\bf k})$ proposed by the first three authors.
Our reformulated conjecture is as follows. 

\begin{conj} (= Conjecture~\ref{reformulated_conjecture}) Given positive integers 
$d$, $g$, and $n$ for which $n \leq 2g \leq d$, and given a ramification profile ${\bf k}=(k_1,\dots,k_n)$ for which the corresponding Severi variety $\mc{V}_{\bf k} \sub M^n_{d,g}$ is nonempty, we have
\begin{equation}\label{general_cod_estimateBis1}
{\rm cod}(\mc{V}_{\bf k},M^n_d) =\sum_{i=1}^n(k_i-i)+\sum_{s\in B} \phi(s)\rho(s)-\sum_{s\in{\bf k}^{\bullet}} \rho(s)- 1
\end{equation}
in which $B$ is the set of Betti elements (see Definition~\ref{Chapman}) of the set of ramification orders $k_i$, $i=1,\dots,n$ together with the minimal generators of $\sss$ less than the conductor $c$; $\rho(s)$ denotes the number of elements of ${\mb N} \setminus {\rm S}$ strictly larger than a given element $s \in {\rm S}$; $\phi(s)$ is the number of factorizations of elements contributed by a given Betti element $s$ that do not arise from strictly smaller Betti elements \footnote{Strictly speaking, this number is relative to a fixed reference factorization associated with the Betti element $s$.}; and ${\bf k}^{\bullet}$ is a distinguished proper subset of those minimal generators less than the conductor $c$ and not in ${\bf k}$ (see the discussion preceding Conjecture~\ref{reformulated_conjecture}). 
\end{conj}

Betti elements are precisely those elements that control the factorizations of elements in ${\rm S}$ as nonnegative linear combinations of minimal generators and elements of the ramification profile. Zeroing in on these is also practically useful; indeed, the new conjecture posits that to compute the codimension of $\mc{V}_{\bf k}$ we may restrict our attention to those elements of $\sss$ having at least two 
decompositions whose associated factorization vectors are orthogonal. 

\medskip
Finally, Section~\ref{supersymmetric_semigroups} is devoted to the case of supersymmetric semigroups with supersymmetric ramification profiles. We prove the following result.

\begin{thm} (= Theorem~\ref{Supersymmetric})
Let $\mathcal{V}_{\bf k} \subset M^3_{d,g; {\rm S}}$ be the subvariety of maps $f: \mb{P}^1 \rightarrow \mb{P}^3$ with a unique singularity that is cuspidal with semigroup ${\rm S}=\langle ab,ac,bc \rangle$ and ramification profile ${\bf k}=(ab,ac,bc)$, where $a$, $b$, and $c$ are pairwise relatively prime positive integers. 
Assume $d=\deg(f) \geq 
2g$; then
\begin{equation}\label{general_cod_estimateBis}
{\rm cod}(\mc{V}_{\bf k},M^3_d) 
=2\rho(abc)+ab+ac+bc-7
\end{equation}
and $\mc{V}_{\bf k}$ is unirational.
\end{thm}

\subsection{Conventions}
We work over $\mb{C}$. By {\it rational curve} we always mean a projective curve of geometric genus zero. A {\it cusp} is a unibranch curve singularity. We let $M^n_d$ denote the space of nondegenerate morphisms $f: \mb{P}^1 \ra \mb{P}^n$ of degree $d>0$. Any such morphism is defined by a linear series $(\mathcal{O}_{\mathbb{P}^1}(d),V)$ with $V\subset H^0(\mathcal{O}_{\mathbb{P}^1}(d))$ and $\dim(V)=n+1$, so hereafter we will identify the morphism with its associated set of homogeneous parameterizing polynomials. We let $M^n_{d,g} \subset  M^n_d$ denote the subvariety of morphisms whose images have arithmetic genus $g$.

\medskip
Every cusp $P$ admits a local parameterization $f: t \mapsto (f_0(t), \dots, f_n(t))$ that is dual to a map of rings $\phi: R \ra \mb{C}[[t]]$, where $R:= \mb{C}[x_0,\dots,x_n]$ and $\phi(x_i):=f_i(t)$. Let $v:\mb{C}[[t]] \ra \mb{N}$ denote the standard valuation induced by the assignment $t \mapsto 1$. Let ${\rm S}:=v(\phi(R))$ denote the {\it value semigroup} of $P$; the complement $\mb{N}\setminus{\rm S}$ is finite, and the (local) {\it genus} of the singularity at $P$ is its cardinality $\delta_P:=\#(\mb{N}\setminus{\rm S})$. The (global arithmetic) genus of $C$ is the sum $g=\sum_{P\in C}\delta_P$ of these local contributions.

\medskip
The {\it conductor} $c$ of a numerical semigroup ${\rm S}$ is $\ell$+1, where $\ell \in {\bf N}\setminus {\rm S}$ denotes the largest gap. A numerical semigroup ${\rm S}$ is {\it hyperelliptic} whenever $2 \in {\rm S}$. It is {\it supersymmetric} whenever it is minimally generated by products $\frac{\prod_{i=1}^n a_i}{a_i}$, $i=1,\dots,n$ of pairwise relatively prime positive integers $a_i$.

\medskip
In this paper, we frequently reference the {\it factorizations} of an element of a numerical semigroup ${\rm S}$ as an $\mb{N}$-linear combination of a finite subset $H$ of elements in ${\rm S}$. In practice, $H$ will typically be the union of the components of vectors ${\bf k}$ and ${\bf k}^*$ specified below. In such instances we abusively write $H= {\bf k} \cup {\bf k}^*$.

\begin{rem}
\emph{Every morphism $f=(f_0:f_1:\ldots:f_n): \mb{P}^1 \ra \mb{P}^n$, dehomogenized and suitably renormalized, is of the form
\begin{equation}
\label{equgrl0}
f_i=t^{k_i}+\sum_{\ell=k_i+1}^{d}b_{i,\ell} \, t^{\ell}
\end{equation}
with respect to a uniformizer $t$ for the integral closure of the local ring $\op$ in $P$.}

\emph{Replacing each $f_i$, $i \geq 1$ by its power series representation $f_i/f_0$ in the preimage of a cusp, we write
\begin{equation}
\label{equgrl}
f_i=t^{k_i}+\sum_{\ell=k_i+1}^{\infty}a_{i,\ell} \, t^{\ell}
\end{equation}
in which our re-use of $t$ is abusive. Note that for any $\ell\leq d$, we have 
\begin{equation}
\label{equail}
a_{i,\ell}=b_{i,\ell}+h
\end{equation}
in which $h$ is a polynomial in the $b_{j,k}$, for $j=0,i$ and $1\leq k\leq \ell-1$.} 

\medskip
\emph{We will see that those independent conditions on coefficients imposed by a cusp on the image of $f$ are of the form $a_{i,\ell}-h'=0$, in which $h'$ is a polynomial in the $a_{j,k}$, 
each condition is associated with a distinct variable $a_{i,\ell}$. The equation~\eqref{equail} then allows us to infer algebraic conditions on the coefficients of the original parameterization $f$.
Because we assume that $d \geq c$, each gap less than $c$ contributes conditions of this type; while no conditions are contributed by terms of order strictly greater than $c$ in the expansions of the $f_i$. The upshot is that for the sake of counting algebraically independent conditions, 
we may replace \eqref{equgrl0} by its local analogue \eqref{equgrl}.}

\end{rem}

\section{Certifying conditions imposed by hyperelliptic cusps}\label{certifying_hyperelliptic_conditions}
In this section, we prove the following result conjectured in \cite{CLM} for unicuspidal rational curves with hyperelliptic cusps.

\begin{thm}\label{theo2.1}
Given positive integers $d$, $g$, and $n$ for which $n \leq 2g \leq d$, the subvariety $\mathcal{V} \subset M^n_{d,g}$ parameterizing rational curves with a unique singularity of hyperelliptic cuspidal type is of codimension
\[
{\rm cod}(\mathcal{V},M^n_d) = (n-1)g.
\]
\end{thm}

In order to prove Theorem~\ref{theo2.1}, we will fix ramification profiles. Accordingly, given a vector ${\bf k}:=(k_1,\ldots,k_n)\in\mathbb{N}_{>0}^{n}$, let $\mathcal{V}_{\bf{k}} \subset M^n_{d,g}$ denote the subvariety parameterizing rational curves $f: \mb{P}^1 \ra \mb{P}^n$ with a unique cusp $P$ for which $(0,k_1,\dots,k_n)$ is the nondecreasing sequence of vanishing orders of sections of $f$ in $f^{-1}(P)$. Given such a map $f$, the associated {\it ramification index} is $r_P:=\sum_{i=1}^n (k_i-i)$, while the number of {\it conditions beyond ramification} $b_P$ is prescribed by
%
\[
{\rm cod}_{[f]}(\mathcal{V}_{\bf k},M^n_d) = r_P+b_P-1
\]
where ${\rm cod}_{[f]}$ is the codimension of the largest component of $\mc{V}_{\bf k}$ containing $[f]$.

\medskip
Hereafter we assume that $k_1=2$. As explained in \cite[Thm. 2.1]{CLM}, we then have
\begin{equation}\label{equbbp}
b_P \geq \sum_{i=2}^{n}\left(g-\frac{k_i}{2} \right)
\end{equation}
and it follows that
\begin{equation} \label{eqcod}
    {\rm cod}(\mathcal{V}_{{\bf k}},M^n_d) \geq (n-1)g+\sum_{i=2}^n\left(\frac{k_i}{2}- i\right).
\end{equation}

The right-hand side of \eqref{eqcod} is minimized when $k_i=2i$, in which case 
its value is $(n-1)g$. So hereafter we assume that $k_i=2i$ for every $i$, 
in which case it suffices to prove that
\begin{equation} \label{b_Pequation}
b_P=\sum_{i=2}^{n}\left(g-i \right).    
\end{equation}

Our proof of (\ref{b_Pequation}) complements the argument of \cite[Thm. 2.1]{CLM}. More precisely, we will certify that the algebraic conditions listed there are exhaustive, by systematically determining all valuations that are realized by elements of the local algebra of a hyperelliptic cusp with ramification profile $(2,4,\dots,2n)$.

Any $F \in \op$ eligible to produce a condition decomposes as a sum of monomials in the parameterizing functions $f_i$, each of which we may assume has $t$-valuation strictly less than $2g$. In other words, 
we have 
\begin{equation} \label{extendedsum}
\begin{split}
F&=\alpha_2f_2+\alpha_{1^2}f_1^{2} \\
&+\alpha_3f_3+\alpha_{2,1}f_2f_1+\alpha_{1^3}f_1^{3} \\
&+\alpha_4f_4+\alpha_{3,1}f_3f_1+\alpha_{2,1^2}f_2f_1^2+\alpha_{2^2}f_2^2+\alpha_{1^4}f_1^{4} \\
& \vdots \\ 
&+\alpha_nf_n+\alpha_{n-1,1}f_{n-1}f_1+\alpha_{n-2,1^2}f_{n-2}f_1^2 +\alpha_{n-2,2}f_{n-2}f_2+\ldots+\alpha_{1^{n}}f_1^{n} \\ 
&+\alpha_{n,1}f_nf_1+\alpha_{n-1,1^2}f_{n-1}f_1^2+\alpha_{n-1,2}f_{n-1}f_2+
\ldots+\alpha_{1^{n+1}}f_1^{n+1} \\
& \vdots \\
&+\alpha_{n,1^p}f_{n}f_1^p+\alpha_{n,2,1^{p-1}}f_{n}f_2 f_1^{p-1}
+\ldots+\alpha_{1^{n+p}}f_1^{n+p}
\end{split}
\end{equation} in which $2g-2=2(n+p)$. The first $n-1$ lines in \eqref{extendedsum} are sums of terms of order $2i$, for $2\leq i\leq n$, while the last $p$ lines are sums of terms of order $2(n+j)$, for $1\leq j\leq p$. 
With some additional notation, we may streamline \eqref{extendedsum} significantly. Namely, let $\ell \in \mathbb{N}$ and let
\[
\mathcal{P}_\ell=\mathcal{P}_{\ell}^n:=\{ i_1\ldots i_s\,|\, i_1+ \cdots+i_s= \ell, \ \text{for}\ i_q \in \{1,\ldots, n \}\}
 \footnote{Hereafter, whenever there is no ambiguity, we will omit commas between parts of partitions.}
 \]
denote the subset of all partitions of $\ell$ each of whose parts is of size at most $n$. Let 
$$
\mathcal{P}_{\ell}^{\circ}:=\mathcal{P}_{\ell} \backslash \{1^{\ell}\} \ \ \textrm{and} \ \ f_{\lambda}:=f_{i_1} \cdots f_{i_s}, \ \lambda=i_1 \ldots i_s \in \mathcal{P}_\ell.
$$
Then (\ref{extendedsum}) may be rewritten succinctly as the statement that
\[
F=\sum_{\ell=2}^{n+p}\left( \sum_{\lambda \in \mathcal{P}_\ell} \alpha_\lambda f_\lambda \right).
\] 
Now let  $\lambda=i_1 \ldots i_s \in \mathcal{P}_{\ell}^{\circ}$ and $i_\lambda:=i_1+\cdots+i_s$. Define recursively 
\begin{equation}
\label{equf*r} 
\begin{split} 
F_{\lambda,0}^{*}&:=f_{\lambda} \\ 
F_{\lambda,r}^{*}&:=F_{\lambda,r-1}^{*}-a^{\lambda,r-1}_{2(i_\lambda+r-1)}f_1^{i_\lambda+r-1} \, \in \,  \oo_P\ \text{for}\ 1\leq r\leq g-i_\lambda
\end{split} 
\end{equation}
 where $a^{\lambda,r}_{m} \in \mathbb{C}[a_{i,\ell}]$ is the coefficient of $t^{m}$ in the power series $F_{\lambda,r}^{*}=F_{\lambda,r}^{*}(t)$.
 
 Note that $a^{\lambda,0}_{2i_\lambda}=1$ for every $\lambda \in \mathcal{P}_\ell^{\circ}$. Moreover, for every $r \geq 1$ we have $a^{\lambda,r}_{2i_{\lambda}+2r-1}=0$ since $2i_\lambda+2r-1 \not\in \sss$. It follows that the vanishing of the coefficients $a^{\lambda,r}_{2i_\lambda+2r-1}$ comprise all conditions beyond ramification; while the conditions that appear in \cite{CLM} are precisely those arising from $F_{i,r}^{*}$, for $2 \leq i \leq n$. So it remains to prove that the latter conditions are exhaustive. A crucial step in this direction is the following.

\begin{rem}
\emph{In the interest of brevity, hereafter 
we will use $a(\lambda,r)$ instead of $a^{\lambda, r}_{2(i_\lambda+r)}$ to denote the coefficient of $t^{2(i_\lambda+r)}$, i.e., of the term that computes the $t$-valuation of $F^*_{\lambda, r}$ in the hyperelliptic case.}
\end{rem}

\begin{lemma}\label{lemma2.2}
Given $\lambda=i_1 \ldots i_s \in \mathcal{P}_\ell^{\circ}$ and $1\leq q< s$, let $\lambda_q:=i_1 \ldots i_q$. We have \begin{equation*}\label{formulaF*}
F^{*}_{\lambda,r}=F_{i_1,r}^{*} \, f_1^{i_{\lambda}-i_{1}}+ \sum_{q=2}^{s}\left( f_{\lambda_{q-1}}F_{i_q,r}^{*} \, f_1^{i_{\lambda}-i_{\lambda_q}}+f_1^{i_{\lambda}-i_{\lambda_{q-1}}}\sum_{l=1}^{r-1}a(i_q,l) \, F^{*}_{\lambda_{q-1},r-l} \, f_1^{l} \right). 
\end{equation*}
In particular, the $F_{\lambda,r}^{*}$ are $\mathbb{C}[f_1,\ldots,f_n]$-linear combinations of the $F_{i_q,u}^{*}$ (associated with singleton partitions) for $1 \leq u\leq r$.

\end{lemma}

\begin{proof}
We proceed by induction on $r$. To begin, assume $r=1$. Note that
\[
\setlength{\jot}{10pt}
\begin{split}
F_{i_1 \ldots i_s,1}^{*}=&f_{i_1} \cdots f_{i_s}-f_1^{i_1+\cdots+i_s} \\
=& f_{i_1} \cdots f_{i_s}-f_{i_1}f_{1}^{i_2+\cdots+i_s}+f_{i_1}f_{1}^{i_2+\cdots+i_s}-f_1^{i_1+i_2+\cdots+i_s} \\
=&F_{i_1,1}^{*} f_{1}^{i_2+\cdots+i_s}+f_{i_1} \left(f_{i_2} \cdots f_{i_s}-f_{1}^{i_2+\cdots+i_s} \right).
\end{split}
\]
Applying the same idea to successively rewrite  $f_{i_j} \cdots f_{i_s}-f_{1}^{i_j+\cdots+i_s}$ for every $j=2,\dots,s-1$, we deduce that
\[
\begin{split}
F_{i_1 \ldots i_s,1}^{*}=&F^{*}_{i_1,1} \, f_1^{i_2+\ldots+i_s}+f_{i_1} F^{*}_{i_{2},1} \, f_1^{i_3+ \ldots +i_{s}}+\cdots+f_{i_1 \ldots i_{s-1}} \, F^{*}_{i_s,1} \\
=& F^{*}_{i_1,1} \, f_1^{i_{\lambda}-i_{1}} +\sum_{q=2}^{s} f_{\lambda_{q-1}}F_{i_q,1}^{*} \, f_1^{i_\lambda-i_{\lambda_{q}}} 
\end{split}
\]
as desired.

\medskip

\noindent Now assume the desired result is valid for some $r > 1$; that is, 
\begin{equation}\label{hypothesis} \Scale[0.95]{ F^{*}_{\lambda,r}=F_{i_1,r}^{*} \, f_1^{i_\lambda-i_{1}}+ \sum_{q=2}^{s}\left( f_{\lambda_{q-1}} F_{i_q,r}^{*} \, f_1^{i_\lambda-i_{\lambda_q}}+f_1^{i_\lambda-i_{\lambda_{q-1}}}\sum\limits_{l=1}^{r-1}a(i_q,l) \, F^{*}_{\lambda_{q-1},r-l} \, f_1^{l} \right).}
\end{equation} 

Note that all terms on the right-hand side of \eqref{hypothesis} are of order $2(i_{\la}+r)$; summing lowest-order coefficients, it follows that
\begin{equation}
\label{formulacoef}
a(\lambda,r)=a(i_1,r)+\sum_{q=2}^{s}\left(a(i_q,r)+\sum_{l=1}^{r-1}a(i_q,l) \cdot a(\lambda_{q-1},r-l) \right).
\end{equation}

\noindent Now let
\[
\begin{split}
X&:=F^{*}_{i_1,r} \, f_1^{i_\lambda-i_{1}}-a(i_1,r) \, f_1^{i_{\lambda}+r}, \,\, \\[3ex]
Y_q&:=f_{\lambda_{q-1}} F^{*}_{i_q,r} \,f_1^{i_{\lambda}-i_{\lambda_q}}-a(i_q,r) \,f_1^{i_{\lambda}+r}, \text{ and }\\[3ex]
Z_q&:=f_1^{i_{\lambda}-i_{\lambda_{q-1}}}\sum_{l=1}^{r-1}Z_q^l, \text{ where } \\[3ex] 
Z_q^l&:=a(i_q,l) \, F^{*}_{\lambda_{q-1},r-l} \, f_1^{l}-a(i_q,l) \cdot a(\lambda_{q-1},r-l) \, f_1^{i_{\lambda_{q-1}}+r-l} \,f_1^{l}.
\end{split}
\]

\medskip
Applying \eqref{hypothesis} together with \eqref{formulacoef} and the definition of $F^{*}_{\lambda,r+1}$, we deduce that
\[
F^{*}_{\lambda,r+1}=F^*_{\lambda,r}-a(\la,r) f_1^{i_{\la}+r}=X+\sum\limits_{q=2}^{s}(Y_q+Z_q). 
\]
In light of the facts that
\[
\Scale[0.9]{\begin{split}
X &=\left(F_{i_1,r}^{*}-a(i_1,r)f_1^{i_1+r} \right)f_1^{i_\lambda-i_{1}}=F_{i_1,r+1}^{*} f_1^{i_\lambda-i_{1}}, \\[3ex]
Y_q &=f_{\lambda_{q-1}}F^{*}_{i_q,r} \, f_1^{i_{\lambda}-i_{\lambda_q}}-f_{\lambda_{q-1}} \, a(i_q,r) \, f_1^{i_{\lambda}-i_{\lambda_{q-1}}
+r}+f_{\lambda_{q-1}} \, a(i_q,r) \, f_1^{i_{\lambda}-i_{\lambda_{q-1}}+r} -a(i_q,r) \, f_1^{i_\lambda+r} \\[3ex] &=f_{\lambda_{q-1}}\left(F_{i_q,r}^{*}-a(i_q,r) f_1^{i_q+r} \right)f_1^{i_\lambda-i_{\lambda_q}} +a(i_q,r) \left(f_{\lambda_{q-1}}-f_1^{i_{\lambda_{q-1}}}\right)f_1^{i_\lambda-i_{\lambda_{q-1}}+r} \\[3ex]
&=f_{\lambda_{q-1}}F_{i_q,r+1}^{*} \, f_1^{i_\lambda-i_{\lambda_q}}+f_1^{i_\lambda-i_{\lambda_{q-1}}}a(i_q,r) \, F_{\lambda_{q-1},1}^{*} \,f_1^r, \text{ and} \\[3ex]
Z_q^l&=a(i_q,l) \, F^{*}_{\lambda_{q-1},r-l} \, f_1^{l}-a(i_q,l) \cdot a(\lambda_{q-1},r-l) \, f_1^{i_{\lambda_{q-1}}+r-l} \, f_1^l \\[3ex]
&= a(i_q,l) \, F^{*}_{\lambda_{q-1},r+1-l} \, f_1^{l}
\end{split}  }
\]

it follows that
\[ 
\Scale[0.82]{\begin{split}
F^{*}_{\lambda,r+1}
&=F^{*}_{i_1,r+1} \, f_1^{i_{\lambda}-i_{\lambda_1}}\\
&+\sum_{q=2}^{s}\bigg( f_{\lambda_{q-1}}\, F^{*}_{i_q,r+1} \, f_1^{i_{\lambda}-i_{\lambda_q}}+f_1^{i_{\lambda}-i_{\lambda_{q-1}}}\, a(i_q,r) \,{F^{*}_{\lambda_{q-1},1} \, f_1^{r}}+f_1^{i_{\lambda}-i_{\lambda{q-1}}}\sum_{l=1}^{r-1} a(i_q,l) F^{*}_{\lambda_{q-1},r+1-l} \, f_1^{l}\bigg) \\
&=F_{i_1,r+1}^{*} \, f_1^{i_\lambda-i_{\lambda_1}}+ \sum_{q=2}^{s}\left( f_{\lambda_{q-1}} \, F_{i_q,r+1}^{*} \, f_1^{i_\lambda-i_{\lambda_q}}+f_1^{i_\lambda-i_{\lambda_{q-1}}}\sum_{l=1}^{r} a(i_q,l) \, F^{*}_{\lambda_{q-1},r+1-l} \, f_1^{l} \right)
\end{split} }
\]
\end{proof}

\begin{ex} \label{example2.2}
\emph{
Assume $n=4$, $g=7$, and $(k_1,k_2,k_3,k_4)=(2,4,6,8)$. Let $F$ be as in \eqref{extendedsum}. Here $p=2$ and $F=\sum_{\ell=2}^{6}\left( \sum_{\lambda \in \mathcal{P}_\ell} \alpha_\lambda f_\lambda \right)$.}
\emph{Because $5$ is a gap of ${\rm S}=\langle 2,15 \rangle$, the first conditions beyond ramification arise from cancelling the terms of order $4$ in the expansion of $F$. Accordingly, we require that
\begin{equation}\label{equfr1}
\alpha_{1^2}=-\alpha_2. 
\end{equation} 
Let $F_1$ denote the function obtained from $F$ after substituting $-\al_2$ for $\alpha_{1^2}$ as in \eqref{equfr1}; we then have 
\[
\begin{split}
F_1
&=\alpha_2F_{2,1}^{*}+R_1
\end{split}
\]
}
\hspace{-11pt} \emph{where $R_1=\sum_{\ell=3}^{6}\left( \sum_{\lambda \in \mathcal{P}_\ell} \alpha_\lambda f_\lambda \right)$ (of order $6$) comprises the terms of $F$ unaffected by the substitution \eqref{equfr1}. The first condition beyond ramification is thus $a_{2,5}-2a_{1,3}=0$, in which we have used
$$
a^{2,1}_{5}=a_{5}^{2,0}-a_{5}^{1^2,0}=a_{2,5}-2a_{1,3}=0.
$$
In the next step we derive conditions by cancelling the terms of order 6 of $F_1$, imposing that 
$
\alpha_{1^3}=-(\alpha_3+\alpha_{21}+\alpha_2 \, a^{2,1}_{6})
=-\sum_{\ell=2}^{3} \left(\sum_{\lambda \in \mathcal{P}_\ell^\circ} \alpha_\lambda a^{\lambda,3-\ell}_{6} \right)$. 
Substituting in $F_1$ yields 
\begin{equation}
\begin{split}
F_2
&=\alpha_3F_{3,1}^{*}+\alpha_{21}F_{21,1}^{*}+\alpha_2 F_{2,2}^{*}+R_2 \\
&=\sum_{\ell=2}^{3}\left(\sum_{\lambda \in \mathcal{P}_\ell^\circ}\alpha_\lambda F_{\lambda,3-(\ell-1)}^{*}\right)+\sum_{\ell=4}^{6} \left( \sum_{\lambda \in \mathcal{P}_{\ell}} \alpha_\lambda f_\lambda \right). \nonumber
\end{split}
\end{equation} 
As $7 \not\in \sss$, a priori we obtain three conditions, namely
\[
{\small \begin{split}
a^{3,1}_{7} &= a_{7}^{3,0}-a_{7}^{1^3,0}=a_{3,7}-3a_{1,3}=0, \\
a^{21,1}_{7} &=a_{7}^{21,0}-a_{7}^{1^3,0}=a_{2,5}-2a_{1,3}=0, \text{ and} \\
a^{2,2}_{7} &=a_7^{2,1}-a_6^{2,1}a_7^{1^3,0}=a_{2,7}-2a_{1,5}+4a_{1,3}a_{1,4}-3a_{1,3}a_{2,6}+3a_{1,3}^3=0.
\end{split}}
\]}

\emph{However, we have $a_{2,5}-2a_{1,3}=0$ from the previous step, so this step produces only two new algebraically independent conditions beyond ramification.}

\medskip
\emph{Next, to cancel the terms of order $8$ of $F_2$, we impose that $$\alpha_{1^4}=-\left(\sum_{\lambda \in \mathcal{P}_4^\circ}\alpha_\lambda a_{8}^{\lambda, 0} +\sum_{\lambda \in \mathcal{P}_3^\circ}\alpha_\lambda a^{\lambda,1}_{8}+\sum_{\lambda \in \mathcal{P}_2^\circ}\alpha_\lambda a^{\lambda,2}_{8} \right)=-\sum_{\ell=2}^{4} \left(\sum_{\lambda \in \mathcal{P}_\ell^\circ} \alpha_\lambda a^{\lambda,4-\ell}_{8} \right).$$ Let $F_3$ denote the function obtained from $F_2$ via substitution; then
\begin{equation}
\begin{split}
F_3
&=\alpha_4F_{4,1}^{*}+\alpha_{31}F_{31,1}^{*}+\alpha_{2^2}F_{2^2,1}^{*}+\alpha_{21^2}F_{21^2,1}^{*} + \alpha_{3}F_{3,2}^{*}+\alpha_{21}F_{21,2}^{*}+\alpha_{2}F_{2,3}^{*}+R_3 \\
&=\sum_{\ell=2}^{4}\left(\sum_{\lambda \in \mathcal{P}_\ell^\circ}\alpha_\lambda F_{\lambda,4-(\ell-1)}^{*}\right)+\sum_{\ell=5}^{6} \left( \sum_{\lambda \in \mathcal{P}_{\ell}} \alpha_\lambda f_\lambda \right). \nonumber
\end{split}
\end{equation} 
It follows necessarily that $a^{\lambda,1}_{9}=0$, $\lambda \in \mathcal{P}_4^\circ$; $a^{\lambda,2}_{9}=0$, $\lambda \in \mathcal{P}_3^\circ$; and $a^{\lambda,3}_{9}=0$, $\lambda \in \mathcal{P}_2^\circ$. Analogously to the previous step, here we have 
\[
\begin{split}
F_{31,1}^{*}&=F_{3,1}^{*} f_1, \, F_{2^2,1}^{*}=f_2^2-f_1^4=
F_{2,1}^{*}f_1^{2}+f_2F_{2,1}^{*}=F_{2,1}^{*}(f_1^{2}+f_2), \\
F_{21^2,1}&=F_{2,1}^{*} f_1^{2}, \text{ and } F_{21,2}^{*}=F_{2,2}^{*}f_1
\end{split}
\]
which yield $a^{31,1}_{9}=a^{3,1}_{7}$, $a^{2^2,1}_{9}=2a^{2,1}_{5}$, $a^{21^2,1}_{9}=a^{2,1}_{5}$, and $a^{21,2}_{9}=a^{2,2}_{7}$, respectively. The only new conditions beyond ramification obtained in this step are $a^{4,1}_{9}=0$, $a^{3,2}_{9}=0$ and $a^{2,3}_{9}=0$.}

\medskip
\emph{The remaining steps follow similarly. Indeed, to cancel terms of order 10 we impose that} 
$\alpha_{1^5}=-\sum_{\ell=2}^{5} \left(\sum_{\lambda \in \mathcal{P}_\ell^\circ} \alpha_\lambda a^{\lambda,5-\ell}_{10} \right) \nonumber 
$
\emph{and we deduce that}
\[
\begin{split}
F_4&=\sum_{\ell=2}^{5}\left(\sum_{\lambda \in \mathcal{P}_\ell^\circ}\alpha_\lambda F_{\lambda,5-(\ell-1)}^{*}\right) + \sum_{\lambda \in \mathcal{P}_{6}^{\circ}}\alpha_\lambda f_\lambda .
\end{split}
\]
\emph{The conditions beyond ramification obtained in this step are $a^{\lambda,1}_{11}=0$, $\lambda \in \mathcal{P}_5^\circ$, $a^{\lambda,2}_{11}=0$, $\lambda \in \mathcal{P}_4^\circ$, $a^{\lambda,3}_{11}=0$, $\lambda \in \mathcal{P}_3^\circ$ and $a^{\lambda,2}_{11}=0$, $\lambda \in \mathcal{P}_2^\circ$; but among these, the only new ones are $a^{4,2}_{11}=0$, $a^{3,3}_{11}=0$ and $a^{2,4}_{11}=0$.}

\medskip
\emph{Finally, substituting
$\alpha_{1^6}=-\sum_{\ell=2}^{6} \left(\sum_{\lambda \in \mathcal{P}_\ell^\circ} \alpha_\lambda a^{\lambda,6-\ell}_{12} \right)
$
we obtain 
\[
\begin{split}
F_5&=\sum_{\ell=2}^{6}\left(\sum_{\lambda \in \mathcal{P}_\ell^\circ}\alpha_\lambda F_{\lambda,6-(\ell-1)}^{*}\right); \nonumber
\end{split}
\] 
the conditions beyond ramification obtained in this step are $a^{\lambda,1}_{13}=0$, $\lambda \in \mathcal{P}_6^\circ$, $a^{\lambda,2}_{13}=0$, $\lambda \in \mathcal{P}_5^\circ$, $a^{\lambda,3}_{13}=0$, $\lambda \in \mathcal{P}_4^\circ$, $a^{\lambda,4}_{13}=0$, $\lambda \in \mathcal{P}_3^\circ$ and $a^{\lambda,5}_{13}=0$, $\lambda \in \mathcal{P}_2^\circ$, but among these, the only new ones are $a^{4,3}_{13}=0$, $a^{3,4}_{13}=0$ and $a^{2,5}_{13}=0$. It follows easily that $b_P=\sum_{i=2}^{4}(7-i)=12$, as expected.}


\end{ex}


\begin{rem}\label{rem2.4}
\emph{A few facts that follow from the proof of Lemma~\ref{lemma2.2} are  worth mentioning, as they are at the heart of the proof of Theorem~\ref{theo2.1}. First, according to \eqref{formulaF*}, the coefficient of $F_{\lambda,r}^{*}$ of order $2(i_{\lambda}+r)+j$, $j \geq 0$ is 
\begin{equation}\label{formulacoef2}
\begin{split}
a^{\lambda,r}_{2(i_\lambda+r)+j}&=a^{i_1,r}_{2(i_1+r)+j}+\sum_{q=2}^{s}\left(a^{i_q,r}_{2(i_q+r)+j}+\sum_{l=1}^{r-1} a(i_q,l) \cdot a^{\lambda_{q-1},r-l}_{2(i_{\lambda_{q-1}}+r-l)+j} \right).
\end{split}
\end{equation}
Arguing recursively, we see that $a^{\lambda,r}_{2(i_{\lambda}+r)+j}$ is a function of $a^{i_q,u}_{2(i_q+u)+j}$, $u \leq r$ and $q \in \{1,\ldots,s\}$.
On the other hand, the order of vanishing of $F_{\lambda,r}^{*}$ is $2(i_{\lambda}+r)$, while by definition
\begin{equation}
 \label{equdf*}  
 F_{\lambda,r}^{*}:=F_{\lambda,r-1}^{*}-a(\lambda,r-1) f_1^{i_\lambda+r-1}
\end{equation}
and the condition produced by $F^{*}_{\lambda,r}$ associated with the gap $2(i_{\lambda}+r)-1$ is $a^{\lambda,r}_{2(i_\lambda+r)-1}=0$. Now \eqref{equdf*} yields $a^{\lambda,r}_{2(i_\lambda+r)-1}=a^{\lambda,r-1}_{2(i_\lambda+r-1)+1}$, and by 
\eqref{formulacoef2} it follows that $a^{\lambda,r-1}_{2(i_\lambda+r-1)+1}$ is a linear combination of the coefficients $a^{i_q,u}_{2(i_q+u)+1}$, which vanish. 
Accordingly we see that $F_{\lambda_q,r}^{*}$ produces no new conditions beyond ramification whenever $2 < q \leq s$.}
\end{rem}


\begin{proof}[Proof of Theorem 2.1]
We follow closely the pattern of Example \ref{example2.2}. 
At each step, indexed by a gap $\ell$ of ${\rm S}$, we (a) we impose a relation among coefficients; (b) substitute correspondingly into an inductively-modified version of $F$; and (c) deduce the conditions beyond ramifications imposed by $\ell$.

\medskip
\noindent So assume we are at step $i$. Starting from
\begin{equation} \label{F-ith}
\begin{split}
F_{i-1}=\sum_{\ell=2}^{i} \left(\sum_{\lambda \in \mathcal{P}_i^{\circ}}\alpha_\lambda F_{\lambda,i+1-\ell}^{*} \right)+\sum_{\lambda \in \mathcal{P}_{i+1}}\alpha_\lambda f_\lambda+\sum_{\ell=i+2}^{n+p} \left(\sum_{\lambda \in \mathcal{P}_\ell}\alpha_\lambda f_\lambda \right)
\end{split}
\end{equation}
we impose the relation 
\begin{equation} \label{alpha_i-th}
\begin{split}
\alpha_{1^{i+1}}=-\sum_{\ell=2}^{i+1} \left(\sum_{\lambda \in \mathcal{P}_\ell^\circ} \alpha_\lambda a^{\lambda,i+1-\ell}_{2 (i+1)} \right)
\end{split}
\end{equation}
which forces the terms of order $2i+2$ to cancel.
Note that an instance of $\alpha_{1^{i+1}}$ appears in the last term of the second summand of the right-hand side of \eqref{F-ith}. Replacing it by the right-hand side of \eqref{alpha_i-th} and labeling the last summand of \eqref{F-ith} by $R_i$, we obtain
\[
\Scale[0.9]{ \begin{split}
F_{i}&=\sum_{\ell=2}^{i} \left(\sum_{\lambda \in \mathcal{P}_i^{\circ}}\alpha_\lambda F_{\lambda,i+1-\ell}^{*} \right)+\sum_{\lambda \in \mathcal{P}_{i+1}^\circ}\alpha_\lambda f_\lambda+\alpha_{1^{i+1}}f_1^{i+1}+R_i \\
&=\sum_{\ell=2}^{i} \left(\sum_{\lambda \in \mathcal{P}_i^{\circ}}\alpha_\lambda F_{\lambda,i+1-\ell}^{*} \right)+\sum_{\lambda \in \mathcal{P}_{i+1}^\circ}\alpha_\lambda f_\lambda-\left[\sum_{\ell=2}^{i+1}\left(\sum_{\lambda \in \mathcal{P}_\ell^\circ} \alpha_\lambda a^{\lambda,i+1-\ell}_{2 (i+1)} \right) \right]f_1^{i+1}+R_i \\ 
&=\sum_{\ell=2}^{i} \left(\sum_{\lambda \in \mathcal{P}_\ell^\circ} \alpha_\lambda \left(F_{\lambda,i+1-\ell}^{*}-a^{\lambda,i+1-\ell}_{2 (i+1)}f_1^{i+1} \right) \right) +\sum_{\lambda \in \mathcal{P}_{i+1}^\circ} \alpha_\lambda \left(f_\lambda-a^{\lambda,0}_{2(i+1)}f_1^{i+1} \right)+R_i. \nonumber
\end{split} }
\]

Note that $2(i+1)=2(i+1-\ell)+2\ell$, while for any $\lambda=i_1 \ldots i_s \in \mathcal{P}_\ell^ \circ$ we have $\ell=i_1+\ldots+i_s$; so $2\ell=2(i_1+\ldots+i_s)=2i_\lambda$. Inasmuch as $k_{i+1}=2(i+1)$ whenever $0 \leq i \leq n-1$ and $k_{n+j+1}=k_n+2(j+1)=2(n+j+1-\ell)+2\ell$ whenever $i=n+j$ with $j \geq 0$, it follows that  
\[
\Scale[0.9]{ \begin{split}
F_i=&\sum_{\ell=2}^{i} \left(\sum_{\lambda \in \mathcal{P}_\ell^\circ} \alpha_\lambda \left(F_{\lambda,i+1-\ell}^{*}-a^{\lambda,i+1-\ell}_{2(i_\lambda+i+1-\ell)}f_1^{i_\lambda+i+1-\ell} \right) \right)+\sum_{\lambda \in \mathcal{P}_{i+1}^\circ} \alpha_\lambda \left(f_\lambda-a^{\lambda,0}_{2 i_\lambda}f_1^{i_\lambda} \right) +R_i \\
=&\sum_{\ell=2}^{i+1} \left(\sum_{\lambda \in \mathcal{P}_\ell^\circ} \alpha_\lambda F_{\lambda,i+2-\ell}^{*} \right) +R_i.
\end{split} }
\]

\noindent After $n+p$ steps, we obtain
$F_{n+p}=\sum_{\ell=2}^{n+p} \left(\sum_{\lambda \in \mathcal{P}_\ell^\circ} \alpha_\lambda F_{\lambda,n+p+1-\ell}^{*} \right)$. According to Lemma~\ref{lemma2.2}, only those functions $F_{\lambda,r}^{*}$ with $\lambda \in \{2,\ldots,n\}$ provide conditions beyond ramification. As each of the first $n-1$ sums of $F_{n+p}$ is associated with a unique function $F_{\lambda,n+p+1-\ell}^{*}$ eligible to produce new conditions beyond ramification, and each such function yields $n+p+1-\ell$ conditions, 
the total number of such conditions obtained is given by
\[
b_P=\sum_{\ell=2}^{n}(n+p+1-\ell)=\frac{1}{2}(n-1)(n+2p).
\]
On the other hand, we have $2g=k_n+2p+2$ and $k_n=2n$, so
\[
b_P=\sum_{\ell=2}^{n}(n+p+1-\ell)=\sum_{\ell=2}^n\left(g-\ell \right)
\]
as desired.

\end{proof}

\begin{coro}\label{hyperelliptic_severi_unirationality}
With the same hypotheses and notation as in Theorem~\ref{theo2.1}, the generic hyperelliptic Severi stratum $\mc{V}_{(2,4,\dots,2n)}$ is unirational.
\end{coro}

\begin{proof}
It follows from the proof of Theorem~\ref{theo2.1} that the conditions imposed by each gap $\ell$ are 
indexed by the singleton partitions $\lambda=2,3,\ldots,(\ell-1)/2$, and are of the form $a^{\lambda,(\ell+1)/2 -\lambda}_{\ell}= a_{\lambda,\ell}+\ldots$; it is the linearity property of the (right-hand sides of) the latter equalities that certifies that $\mc{V}_{(2,4,\dots,2n)}$ is unirational.
\end{proof}

\begin{rem}
\emph{We expect that appropriate adaptations of the proofs of Lemma~\ref{lemma2.2}, Theorem~\ref{theo2.1}, and Corollary~\ref{hyperelliptic_severi_unirationality} extend more generally to cases in which $k_i < 2g$ and $k_i$ is even for every $i$; 
there will then be $b_P= \sum_{i=2}^n (g-\frac{k_i}{2})$ conditions beyond ramification, and correspondingly
\[
{\rm cod}(\mathcal{V}_{{\bf k}},M^n_d)= (n-1)g+\sum_{i=2}^n\left(\frac{k_i}{2}- i\right).
\]}
\end{rem}


\section{Dimensionality (and unirationality) conjectures for Severi varieties, revisited}\label{Severi_dimensionality_conjectures}

In this section we recast the dimensionality conjecture \cite[Conj.~3.6]{CLM} in terms of {\it Betti elements} naturally associated with the pair $({\rm S},\bf{k})$. Betti elements are distinguished by their factorizations, which control the factorizations of all other elements in ${\rm S}$. 


\begin{dfn} (see \cite[Sec. 2]{Chapman})
\label{Chapman}
\emph{For $T=\{t_1,\ldots,t_\ell \}\subset \mathbb{N}$, let $\mathcal{F}(T)$ be the free monoid generated by $T$, and let 
$\sss_T:=\{a_1t_1+\ldots +a_\ell t_\ell \,|\, a_i\in\mathbb{N}\}\subset\mathbb{N}$.
There is a natural projection
\[
\pi : \mathcal{F}(T) \longrightarrow \sss_T.
\]
Given $s\in\sss_T$, let $Z_s:=\pi^{-1}(s)\subset \mathbb{N}^\ell$; given a pair of elements $v,w\in Z_s$, we say $v\sim w$ whenever there is a chain $v=v_1,\ldots,v_n=w$ such that $\langle v_i,v_{i+1}\rangle\neq 0$. Then $\sim$ defines an equivalence relation on $Z_s$, and we say that $s$ is a \emph{Betti element of $T$} whenever $Z_s$ contains at least two $\sim$-equivalence classes.}
\end{dfn}

Now given a numerical semigroup $\sss$ together with an increasing sequence of strictly positive integers ${\bf k}=(k_1,\dots,k_n)$, let ${\bf k^{\ast}}:=\{k_{n+1},\dots,k_{\ell}\}$ be the increasing sequence of minimal generators of ${\rm S}$ strictly less than the conductor $c$ and not equal to any $k_i$, $i=1,\dots,n$. 

\medskip
Let $B$ denote the set of Betti elements of ${\bf k}\cup{\bf k^{\ast}}$ strictly less than $c$.
Given $b\in B$, let $\psi(b)$ denote the number of $\sim$-equivalence classes of $b$. 
Assume $b_1<\dots<b_p$ comprise $B$. Let 
$
V_i:=\{v_{1,i},\ldots,v_{\psi(b_i),i}\}\subset\mathbb{N}^{\ell}
$
denote a full set of representatives of $\sim$-equivalence classes of $b_i \in B$, with $1\leq i \leq p$; and let $V(E)$ denote the vector matroid on
\[
E:=\{v_{1,2}-v_{1,1},\ldots,v_{1,\psi(b_1)}-v_{1,1};\,\ldots\ \ldots\, ; v_{p,2}-v_{p,1},\ldots,v_{p,\psi(b_p)-v_{p,1}}\}. \]

Let $C_1,\ldots,C_q$ denote the circuits of $V(E)$. For every $i=1,\dots,q$, let $b(i)$ be the largest element among $b_1,\dots,b_p$ for which $v_{b(i),j}-v_{b(i),1}\in C_i$ for some $j$.
%
Given $b=b_i\in B$, let
$\phi(b):= \psi(b)-1-\#\{i\,|\, b=b(i)\}$ and $B^{\pr}:=\{ b\in B\,|\,\phi(b)\geq 1\}$.
Further, let
\[
\begin{split}
m:=&\min \{j \, |  \, k_j >  b \text{ for some }  k_j \in {\bf k}^* \text{ and } b \in B'  \}; \text{ and }  \\
{\bf k}^{\bullet}:=&\{ k_{j} \in {\bf k}^* \ |\  B'\cap (k_{j-1},k_{j})\neq\emptyset\text{ or } \sum\limits_{b <k_{j}} \phi(b)  > j-m \}.
\end{split}
\]
Finally, given $s \in \sss$, let $\rho(s):=\#\{r>s: r\not\in {\rm S}\}$. The {\it syzygetic defect} of $\sss$ with respect to ${\bf k}$ is
$D({\rm S},{\bf k}):= \sum_{i=1}^q\rho(b(i))$.

Our adjusted version of \cite[Conj.~3.6]{CLM} reads as follows.


\begin{conj}\label{reformulated_conjecture_bis}
Given positive integers 
$d$, $g$, and $n$ for which $n \leq 2g \leq d$, and a ramification profile ${\bf k}=(k_1,\dots,k_n)$ for which the corresponding Severi variety $\mc{V}_{\bf k} \sub M^n_{d,g}$ is nonempty, we have
\begin{equation}
\label{general_cod_estimate}
{\rm cod}(\mc{V}_{\bf k},M^n_d) =\sum_{i=1}^n(k_i-i)+\sum_{s \in B} (\psi(s)-1) \rho(s)- \sum_{s \in {\bf k}^{\bullet}} \rho(s)- D({\rm S},{\bf k})-1.
\end{equation}
\end{conj}

The following alternative formulation is often computationally useful.

\begin{conj}\label{reformulated_conjecture}
Given positive integers 
$d$, $g$, and $n$ for which $n \leq 2g \leq d$, and a ramification profile ${\bf k}=(k_1,\dots,k_n)$ for which the corresponding Severi variety $\mc{V}_{\bf k} \sub M^n_{d,g}$ is nonempty, we have
\begin{equation}\label{general_cod_estimateBis}
{\rm cod}(\mc{V}_{\bf k},M^n_d) =\sum_{i=1}^n(k_i-i)+\sum_{s\in B} \phi(s)\rho(s)-\sum_{s\in{\bf k}^{\bullet}} \rho(s)- 1.
\end{equation}
\end{conj}

\begin{rem}
\emph{In \cite[Conj.~3.6]{CLM}, the authors used the language of (irreducible) decompositions in place of factorizations of Betti elements. Given $s \in {\rm S}$ and a finite subset $H=\{s_1,\dots,s_h\}$ of ${\rm S}$, we may identify any {\it decomposition} 
\begin{equation}
\label{equdec}
s=m_1 s_1+\ldots+m_h s_h
\end{equation}
of $s$ with respect to $H$ with the vector $v=(m_1,\ldots,m_h)\in\nn^h$, and we say that $v$ is an {\it irreducible} (resp. {\it reducible}) decomposition of $s$ (with respect to $H$) if no (resp. some) proper sub-sum of the right-hand side of (\ref{equdec}) factors as an $\mb{N}$-linear combination of $s_1, \dots, s_h$.} 

\emph{To say that $v$ is a reducible decomposition of $s\in\sss$ means that there is some $w\in\mathbb{N}^h$ for which (i) $v-w$ is a decomposition of some $s^{\pr} <s$ in $\sss$; and (ii) $s^{\pr}$ admits a decomposition distinct from $v-w$, say $u$. Whenever this is the case, $u+w$ is a decomposition of $s$ distinct from $v$. Now assume $H=\bf{k} \cup \bf{k}^*$. In counting conditions imposed by cusps of type $({\rm S},\kk)$, it follows that the vector $v-(u+w)=(v-w)-u$ produced by (the decompositions of) $s$ has already been accounted for by $s^{\pr}$, so it doesn't contribute any further conditions. Similarly, suppose $v_1$ and $v_2$ are factorizations of $s\in {\rm S}$ whose inner product is nonzero; assume that  their $i$-th components are nonzero.
Let $w=e_i$; then $v_1-w$ and $v_2-w$ are factorizations of the smaller element $s^{\pr} =s-k_i\in\sss$. In particular, this means that only elements of $B$ contribute independent conditions, and that each of these conditions is selected for by a distinct equivalence class of factorizations.}

\medskip
\emph{In the current version of our dimension conjecture, the set ${\bf k}^*$ contributes conditions to the codimension not accounted for by \cite[Conj.~3.6]{CLM}. 
An instructive example is that of
$\sss = \langle 20, 22, 24; 45 \rangle \text{ and } {\bf k}=\{20,22,24\}$, in which case the element $s=90 \in \sss$ generates new conditions beyond ramification that are ``invisible" to the previous version of our conjecture, which referenced only factorizations relative to ${\bf k}$.}

\medskip
\emph{Finally, the set ${\bf k}^{\bullet}$ replaces the parameter $m$ of $\sum_{i=1}^{m} \rho(s^{\ast}_i)$ 
in {\it ibid.}, and deserves some explanation. 
For this purpose, we revisit \cite[Example 3.5]{CLM}, in which $g=11$, $n=4$, 
${\bf k}=\{4,8,10,12\}$, and $\sss=\langle 4,10,15\rangle$. The conditions beyond ramification correspond to the light red boxes in the Dyck diagram below}. 

\begin{center}
\includegraphics[scale=0.25]{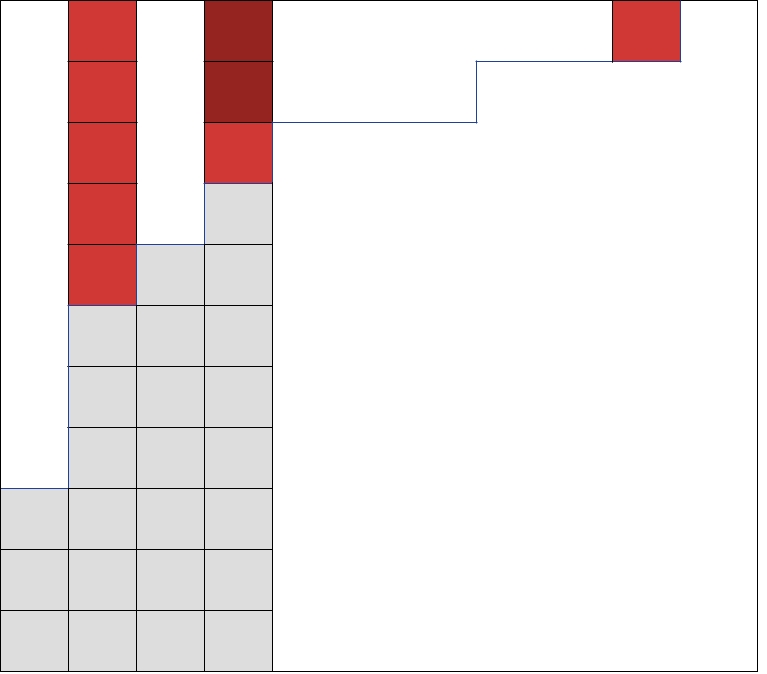} 
\end{center}

\emph{Let $f=(f_1, f_2, f_3, f_4)$ denote a parameterization indexed by an element of the associated Severi stratum. The smallest elements in $B'$ are $8$ and $12$, corresponding to the orders of $f_2$ and $f_4$ (columns $2$ and $4$). We adapt the partition-based notation used previously in the hyperelliptic case as follows. First write $F\in\op$ as in \eqref{extendedsum}; 
and define $F_{\lambda,r}^*$ exactly as (\ref{equf*r}). Algebraic conditions imposed on the coefficients of $f$ are now produced iteratively, 
starting from $F^*_{2,1}$ and $F^*_{4,1}$, reespectively, and are combinatorially witnessed by successively higher and higher light red boxes 
in their respective columns. However, $15\in {\bf k}^*$ (column $6$) ``blocks" both iterative procedures as it is not the order of any monomial in the parameterizing functions $f_i$, $i=1,\dots,4$. As a result, we replace either $F^*_{2,4}$ or $F^*_{4,2}$ by $F^*_{2,4}-a^{2,4}_{15}F^*_{4,2}$ (of order $16$) and proceed with the iterative process, which now combinatorially unfolds along only one of the two columns. This accounts for the dark red boxes at the top of column 4, which index no further algebraically independent conditions, and 
as a result contribute the correction factor $-\rho(15)$ in (\ref{general_cod_estimateBis}). 
On the other hand, an element of $k_i \in \bf{k}^*$ blocks the promotion of an element $b \in B'$ upward in its corresponding column whenever either i) there are fewer elements $k_j$ strictly less than $k_i$ than elements $b \in B'$ or ii) there is at least one $b \in B'$ in the interval $(k_{i-1},k_i)$. This dichotomy accounts for ${\bf k}^{\bullet}$}. 
\end{rem}

\begin{rem}
\emph{Graphically, Betti elements $B=B({\rm S},{\bf k})$ index columns in the Dyck path representation of $({\rm S},{\bf k})$ which conditions beyond ramification may appear.
In the case of the generic hyperelliptic stratum in which ${\rm S}=\langle 2,2g+1 \rangle$ and ${\bf k}=(2,4,\dots, 2n)$, Conjecture~\ref{reformulated_conjecture_bis} specializes to (the dimensionality assertion in) Theorem~\ref{theo2.1}. In general, note that $k_1 \notin B$ whenever $k_1=\text{mult}({\rm S})$; so in such cases, we have $\# (B \cap \{k_1,\dots,k_n\} )\leq n-1$. The generic hyperelliptic stratum thus represents a case that is {\it maximal} with respect to Betti elements; in the following section we will focus on cases that are minimal with respect to $B$.}
\end{rem}

\section{Supersymmetric semigroups}\label{supersymmetric_semigroups}

According to \cite[Thm 12]{GSOR}, supersymmetric semigroups are precisely those numerical semigroups whose sets of Betti elements are singletons; so cases in which ${\rm S}$ is supersymmetric and ${\bf k}$ comprises a set of minimal generators are minimal with respect to $\# B({\rm S},{\bf k})$. In this section, we verify Conjecture~\ref{reformulated_conjecture_bis} when ${\rm S}$ is a supersymmetric semigroup of embedding dimension three, and ${\bf k}$ is its set of minimal generators. Hereafter, we assume $a$, $b$, and $c$ are pairwise relatively prime positive integers. 

\begin{prop} \label{1or3factors}
Let $\sss=\sss(a,b,c)=\langle ab,ac,bc \rangle$ be the supersymmetric semigroup determined by the triple $(a,b,c)$. Every $s\in\sss$ strictly less than the conductor $\mathfrak{c}$ that does not uniquely factor as a nonnegative linear combination of $ab$, $ac$, and $bc$ is of the form
\begin{equation}
\label{equexp}
s=abc+x\cdot ab+y\cdot ac+z\cdot bc
\end{equation}
with $0\leq x\leq c-1$, $0\leq y\leq b-1$ and $0\leq z\leq a-1$. Furthermore, $s$ admits exactly three factorizations, given by $(c+x,y,z)$, $(x,b+y,z)$ and $(x,y,a+z)$.
\end{prop}

\begin{proof}
Suppose that $s=x'\cdot ab+y'\cdot ac+z'\cdot bc$. The Frobenius number of $\sss$ is $2abc - (ab+ac+bc)$; it follows that
\begin{equation}\label{equfrc}   
\frac{x'+1}{c}+\frac{y'+1}{b}+\frac{z'+1}{a} < 2. 
\end{equation}
Since $s$ has more than one factorization, we have either $x'\geq c$, $y'\geq b$, or $z'\geq a$. Suppose, without loss of generality, that $x^{\pr} \geq c$; then \eqref{equfrc} forces $x'<2c$, $0\leq y'\leq b-1$ and $0\leq z'\leq a-1$. Letting $x'=c+x$ with $0\leq x\leq c-1$, $y'=y$, and $z'=z$ yields \eqref{equexp}. As $(x,y,z)$ is the unique factorization of $s-abc$, it follows that $s$ has precisely the three factorizations claimed in the statement of the proposition.
\end{proof}

\begin{thm} \label{Supersymmetric}
Let $\mathcal{V}_{\bf k} \subset M^3_{d,g; {\rm S}}$ be the subvariety of maps $f: \mb{P}^1 \rightarrow \mb{P}^3$ with a unique singularity that is cuspidal with semigroup ${\rm S}=\langle ab,ac,bc \rangle$ and ramification profile ${\bf k}=(ab,ac,bc)$, where $a$, $b$, and $c$ are pairwise relatively prime positive integers. 
Assume $d=\deg(f) \geq 
2g$; then
\begin{equation}\label{general_cod_estimate_abc}
{\rm cod}(\mc{V}_{\bf k},M^3_d) 
=2\rho(abc)+ab+ac+bc-7
\end{equation}
and $\mc{V}_{\bf k}$ is unirational.
\end{thm}

\begin{proof} 

Starting from a parameterization
$
f=(f_0:f_1:f_2:f_3)
$, abusively replace $f_i$ by its local power series representation $f_i/f_0$; 
and let $t$ denote a uniformizer for the local ring $\op$ in $P$.
Any $F \in \op$ eligible to produce a condition decomposes as a sum of monomials in the parameterizing functions $f_i$, each of which we may assume has $t$-valuation greater than or equal to the unique Betti element $abc$, and strictly less than the conductor of ${\rm S}$, so in particular strictly less than $2g$. Collecting monomials in the $f_i$ with common $t$-adic valuations, the corresponding ``universal polynomial" $F$ in $f_1$, $f_2$, and $f_3$ is a sum of monomials indexed by the elements of ${\rm S}$ of the following shape: 
\begin{multicols}{2}
\begin{minipage}{4cm}
\begin{equation*} \label{extendedsum2}
\begin{split}
\text{ Elements of \sss} \\
abc:\\
abc+1: \\
\vdots \ \ \ \ \\
abc+\ell_1-1:\\
abc+\ell_1: \\ 
abc+\ell_1+1:\\
\vdots  \ \ \ \ \\ 
abc+j_1-1: \\
abc+j_1: \\
abc+j_1+1: \\
\vdots \ \ \ \ 
\end{split}
\end{equation*}

\end{minipage}
\begin{minipage}{4cm}
\flushright
\begin{equation*} \label{extendedsum2}
\begin{split}
&\text{ Monomials of $F$} \\
&\alpha_1\,f_1^c+\beta_{1}\,f_2^{b}+\gamma_{1}\,f_3^a \\
&+\sigma_{1}\,h_{1} \\
&\,\, \ \ \ \ \vdots\\
&+\sigma_{\ell_1-1}\,h_{\ell_1-1} \\
&{\text   (gap)}\\ 
&+\sigma_{\ell_1+1}\,h_{\ell_1+1} \\
&\,\, \ \ \ \ \vdots\\ 
&+\sigma_{j_1-1}\,h_{j_1-1}  \\
&+\alpha_2 \,g_{1}\, f_{1}^c+\beta_{2} \,g_{1}\,f_2^{b} + \gamma_{2}\, g_{1}\, f_3^{a} \\ 
&+\sigma_{j_1+1}\,h_{j_1+1}  \\
&\,\, \ \ \ \ \vdots\\ 
\end{split}
\end{equation*}
\end{minipage}
\end{multicols}

Monomials in the first line have $t$-adic valuation $abc$. To usefully describe monomials in the remaining lines, we invoke Proposition~\ref{1or3factors}, and let $h_i$ (resp., $g_i$) denote a monomial with valuation $abc+i$ (resp., $j_i$) for which $abc+i$ (resp., $abc+j_i$) factors uniquely (resp., that admits 3 distinct factorizations); 
and we label those gaps of ${\rm S}$ strictly greater than $abc$ in increasing order by $abc+\ell_i$, $i \geq 1$. We then have 
\[
\begin{split}
f_1^c&:= t^{abc} + \sum\limits_{m=1}^{\infty} a_m \, t^{abc+m}, 
f_2^b:= t^{abc} + \sum\limits_{m=1}^{\infty} b_m \, t^{abc+m},
f_3^a:= t^{abc} + \sum\limits_{m=1}^{\infty} c_m \, t^{abc+m};\\ \text{ and} \\
h_i&:= t^{abc+i} + \sum\limits_{m=1}^{\infty} d_{i,m} \, t^{abc+i+m} \text{ and }
g_i:= t^{j_i} + \sum\limits_{m=1}^{\infty} e_{i,m} \, t^{j_i+m} 
\end{split}
\]
for suitable choices of coefficients $a_m$, $b_m$, $c_m$, $d_{i,m}$ and $e_{i,m}$.


\medskip
Following the basic protocol established in the proof of Theorem~\ref{theo2.1}, we now systematically examine the terms of $F$ order by order, alternating between imposing vanishing on coefficients in orders corresponding to elements of ${\rm S}$, and deducing polynomial (vanishing) conditions imposed by terms whose orders are gaps (i.e., elements of the complement $\mb{N} \setminus {\rm S}$). Beginning with terms of valuation $abc$, we have
\[
F = (\alpha_1+\beta_1+\gamma_1) t^{abc}+ (\text{higher-order terms});
\]
accordingly, we impose that 
\begin{equation}\label{Step1}
\gamma_1=-\alpha_1-\beta_1.
\end{equation}

The coefficient of $t^{abc+1}$ in $F$, on the other hand, is
\begin{equation}\label{CoefStep1-2}
[t^{abc+1}]F=\alpha_1\,a_1+\beta_1\,b_1+\gamma_1\,c_1+\sigma_1=
\alpha_1\,(a_1-c_1)+\beta_1\,(b_1-c_1)+\sigma_1
\end{equation}
in which we have invoked the condition \eqref{Step1} we imposed on the terms of valuation $abc$ in the first step. Accordingly, to move on to terms of higher order, we impose
\[
\sigma_1=-\alpha_1(a_1-c_1)-\beta_1(b_1-c_1).
\]

Note that at every step $i$ for which $abc+i \in \sss$ factors uniquely as a nonnegative linear combination of $ab$, $ac$, and $bc$, our process replaces the coefficient $\sigma_i$ to access step $i+1$, while $\alpha_1$ and $\beta_1$ remain free parameters. 
For example, the coefficient of $t^{abc+2}$ in $F$ is
\[
\begin{split}
[t^{abc+2}]F &=\alpha_1\,a_2+\beta_1\,b_2+\gamma_1\,c_2+\sigma_1\,d_{1,1}+\sigma_2 \\
&=\alpha_1\,(a_2-c_2-d_{1,1}(a_1-c_1))+\beta_1\,(b_2-c_2-d_{1,1}(b_1-c_1))+\sigma_2
\end{split}
\]
in which we have invoked the vanishing of the terms of valuation $abc$ and $abc+1$ in the second line. 
Moving on to terms of higher order, we impose
\[
\sigma_2=-\alpha_1 (a_2-c_2-d_{1,1}(a_1-c_1))-\beta_1 (b_2-c_2-d_{1,1}(b_1-c_1)).
\]

Iterating our procedure, we find that
\[
\begin{split}
[t^{abc+\ell_1-1}]F &= \alpha_1 a_{\ell_1-1}+\beta_1 b_{\ell_1-1}+\gamma_1 c_{\ell_1-1}+ \sum\limits_{i=1}^{\ell_1-2} \sigma_i \, d_{i,\, \ell_1-1-i}+\sigma_{\ell_1-1} \\
&=\alpha_1 \cdot C_{\ell_1-1}^{\alpha_1}+\beta_1 \cdot C_{\ell_1-1}^{\beta_1}+\sigma_{\ell_1-1} 
\end{split}
\]
in which each $C_{i}^j$ is a polynomial in the coefficients of $f_1, f_2$ and $f_3$ that appear in step $i$. 
In particular, at step $\ell_1$, corresponding to the gap $abc+\ell_1$,  we obtain 
\[
[t^{abc+\ell_1}]F=\alpha_1 \cdot C_{\ell_1}^{\alpha_1}+\beta_1 \cdot C_{\ell_1}^{\beta_1} 
\]
and accordingly the first two conditions beyond ramification are
$C_{\ell_1}^{\alpha_1}=0$ and $C_{\ell_1}^{\beta_1}=0$. The upshot is that there are two conditions beyond ramification imposed by each gap $abc+\ell_i$ for $\ell_i<j_1$, and these conditions are 
\begin{equation}
\label{equind}
C_{\ell_i}^{\alpha_1}=a_{\ell_i}-c_{\ell_i}-\sum_{k=1}^{\ell_i-1} d_{k,\ell_i-k} C_{k}^{\alpha_1}\ \ \ \ \text{and}\ \ \ \ C_{\ell_i}^{\beta_1}=b_{\ell_i}-c_{\ell_i}-\sum_{k=1}^{\ell_i-1} d_{k,\ell_i-k} C_{k}^{\beta_1}
\end{equation}
where $ C_{k}^{\alpha_1}=C_{k}^{\beta_1}=0$ for $k=\ell_j<\ell_i$. Explicitly, as a sum of the $(a_k-c_k)$'s we may also write \begin{equation}
\label{equind2}
C_{\ell_i}^{\alpha_1}=\sum_{k=1}^{\ell_i} \nu_{k,\ell_i}(a_k-c_k)\ \ \ \ \text{and}\ \ \ \ C_{\ell_i}^{\beta_1}=\sum_{k=1}^{\ell_i} \nu_{k,\ell_i}(b_k-c_k)
\end{equation}
where $\nu_{\ell_i,\ell_i}=1$, $\nu_{\ell_j,\ell_i}=0$ for $\ell_j<\ell_i$, and recursively
\[
\nu_{k,\ell_i}=\sum_{m=0}^{\ell_i-k-1}-d_{k+m,\ell_i-k-m}\,\nu_{k,k+m}.
\]

Hereafter, at each step indexed by a single gap or by an element of S that factors uniquely, our procedure will exhibit analogous behavior. A new phenomenon occurs 
at step $j_1$, 
indexed by $abc+j_1 \in {\rm S}$, which admits three factorizations. 
More precisely, in step $j_1-1$ we have
\[
\begin{split}
[t^{abc+j_1-1}]F &= \alpha_1 \cdot C_{j_1-1}^{\alpha_1}+\beta_1 \cdot C_{j_1-1}^{\beta_1}+\sigma_{j_1-1} 
\end{split}
\]
while in step $j_1$ 
\[
\begin{split}
[t^{abc+j_1}]F &= \alpha_1\,a_{j_1}+\beta_1\,b_{j_1}+\gamma_1\,c_{j_1}+ \sum\limits_{\substack{i=1 \\i\neq \ell_k}}^{j_1-1} \sigma_i \, d_{i,\, j_1-i} + \alpha_2+\beta_2+\gamma_2
\end{split}
\]
which, upon substituting for $\gamma_1$ and the $\sigma_{i}$ as our procedure dictates, becomes
\[
\begin{split}
[t^{abc+j_1}]F &= \alpha_1\,C_{j_1}^{\alpha_1}+\beta_1\,C_{j_1}^{\beta_1}+ \alpha_2+\beta_2+\gamma_2.
\end{split}
\]
In step $j_1+1$, we have
{\small
\[
\begin{split}
[t^{abc+j_1+1}]F &= \alpha_1\,a_{j_1+1}+\beta_1\,b_{j_1+1}+\gamma_1\,c_{j_1+1}+ \sum\limits_{\substack{i=1 \\i\neq \ell_k}}^{j_1-1} \sigma_i \, d_{i,\, j_1-i+1} \\ 
&+ \alpha_2\,(a_1+e_{1,1})+\beta_2\,(b_1+e_{1,1})+\gamma_2\,(c_1+e_{1,1})+\sigma_{j_1+1}.
\end{split}
\]
}
\medskip
Imposing $\gamma_2=-\alpha_1\,C_{j_1}^{\alpha_1}-\beta_1\,C_{j_1}^{\beta_1}-\alpha_2-\beta_2$, and substituting as usual for $\gamma_1$ and $\sigma_i$, $i=1,\dots,j_1-1$, we deduce that 
\medskip
\begin{equation}\label{CoefStep(j_1+1)}
\begin{split}
[t^{abc+j_1+1}]F &= \alpha_1\,C_{j_1+1}^{\alpha_1}+\beta_1\,C_{j_1+1}^{\beta_1}+ \alpha_2\,C_{j_1+1}^{\alpha_2}+\beta_2\,C_{j_1+1}^{\beta_2}+\sigma_{j_1+1}
\end{split}
\end{equation}
%
\medskip
in which $C_{j_1+1}^{\alpha_2}=a_1-c_1$ and $C_{j_1+1}^{\beta_2}=b_1-c_1$.

\medskip

The polynomial structure of \eqref{CoefStep(j_1+1)} reproduces that of \eqref{CoefStep1-2} in step 1; in particular, we have $C^{\al_2}_{j_1+1}=C^{\al_1}_1$ and $C^{\be_2}_{j_1+1}=C^{\be_1}_1$. 
Similarly, in step $j_1+2$, we have
\[
\begin{split}
[t^{abc+j_1+2}]F &= \alpha_1\,a_{j_1+2}+\beta_1\,b_{j_1+2}+\gamma_1\,c_{j_1+2}+ \sum\limits_{\substack{i=1 \\i\neq \ell_k}}^{j_1-1} \sigma_i \, d_{i,\, j_1-i+2} \\
&+ \alpha_2\,(a_2+a_1\,e_{1,1}+e_{1,2})+\beta_2\,(b_2+b_1\,e_{1,1}+e_{1,2})+\gamma_2\,(c_2+c_1\,e_{1,1}+e_{1,2}) \\
&+\sigma_{j_1+1}\,d_{j_1+1,1}+\sigma_{j_1+2}
\end{split}
\] which, upon substituting for $\gamma_1$, $\gamma_2$ and the $\sigma_{i}$ as our method requires, becomes
\[
\begin{split}
[t^{abc+j_1+2}]F &= \alpha_1\,C_{j_1+2}^{\alpha_1}+\beta_1\,C_{j_1+2}^{\beta_1}+ \alpha_2\,C_{j_1+2}^{\alpha_2}+\beta_2\,C_{j_1+2}^{\beta_2}+\sigma_{j_1+2}
\end{split}
\]
where 
\[
C_{j_1+2}^{\alpha_2}=(a_2-c_2-(e_{1,1}-d_{j_1+1,1})(a_1-c_1)) \text{ and }
C_{j_1+2}^{\beta_2}=(b_2-c_2-(e_{1,1}-d_{j_1+1,1})(b_1-c_1)).
\]

\medskip
On the other hand, because 
$h_1\,g_1$ is of order $abc+j_1+1$ (and the coefficient of $t^{abc+j_1+1}$ in its power series expansion is 1), we have $h_1\,g_1=h_{j_1+1}$, which implies that $d_{j_1+1,1}=d_{1,1}+e_{1,1}$. It follows that
\[
C_{j_1+2}^{\alpha_2}=a_2-c_2-d_{1,1}(a_1-c_1)=C_{2}^{\alpha_1} \text{ and }
C_{j_1+2}^{\beta_2}=b_2-c_2-d_{1,1}(b_1-c_1)=C_{2}^{\beta_1}.
\]

Iterating the argument yields $C^{\al_2}_{j_1+i}=C^{\al_1}_{i}$ and $C^{\be_2}_{j_1+i}=C^{\be_1}_{i}$ for every $i$.

\medskip
Note that every gap greater than $abc+j_1$ must be of the form $abc+j_1+\ell_i$, where $\ell_i$ is a gap (though the converse is not true in general). 
Now suppose that $abc+j_1+\ell_i$ is a gap. We then have 
\[
\begin{split}
[t^{abc+j_1+\ell_i}]F &= \alpha_1\,C_{j_1+\ell_i}^{\alpha_1}+\beta_1\,C_{j_1+\ell_i}^{\beta_1}+ \alpha_2\,C_{j_1+\ell_i}^{\alpha_2}+\beta_2\,C_{j_1+\ell_i}^{\beta_2}\\
&= \alpha_1\,C_{j_1+\ell_i}^{\alpha_1}+\beta_1\,C_{j_1+\ell_i}^{\beta_1}+ \alpha_2\,C_{\ell_i}^{\alpha_1}+\beta_2\,C_{\ell_i}^{\beta_1}
\end{split}
\]
which yields precisely two new conditions, $C_{j_1+\ell_i}^{\alpha_1}=0$ and $C_{j_1+\ell_i}=0$, as $C_{\ell_i}^{\alpha_1}$ and $C_{\ell_i}^{\beta_1}$ already vanish.
On the other hand, whenever $abc+j_1+\ell_i$ belongs to ${\rm S}$, we have
\[
\begin{split}
[t^{abc+j_1+\ell_i}]F &= \alpha_1\,C_{j_1+\ell_i}^{\alpha_1}+\beta_1\,C_{j_1+\ell_i}^{\beta_1}+ \alpha_2\,C_{j_1+\ell_i}^{\alpha_2}+\beta_2\,C_{j_1+\ell_i}^{\beta_2}+\sigma_{j_1+\ell_i}\\
&= \alpha_1\,C_{j_1+\ell_i}^{\alpha_1}+\beta_1\,C_{j_1+\ell_i}^{\beta_1}+ \alpha_2\,C_{\ell_i}^{\alpha_1}+\beta_2\,C_{\ell_i}^{\beta_1}+\sigma_{j_1+\ell_i}\\
&= \alpha_1\,C_{j_1+\ell_i}^{\alpha_1}+\beta_1\,C_{j_1+\ell_i}^{\beta_1}+\sigma_{j_1+\ell_i}.
\end{split}
\]
It then follows that
\[
\begin{split}
[t^{abc+j_1+\ell_i+1}]F &= \alpha_1\,C_{j_1+\ell_i+1}^{\alpha_1}+\beta_1\,C_{j_1+\ell_i+1}^{\beta_1}+ \alpha_2\,C_{\ell_i+1}^{\alpha_1}+\beta_2\,C_{\ell_i+1}^{\beta_1}+\sigma_{j_1+\ell_i+1}
\end{split}
\]
and so on. At the conclusion of step $j$, the upshot is that 
\[
[t^{abc+j}]F = \alpha_1\,C_{j}^{\alpha_1}+\beta_1\,C_{j}^{\beta_1}+x
\]
where
\[x=\begin{cases}0 \text{ whenever $abc+j$ is a gap;} \\
\sum\limits_{i=2}^{k} (\alpha_i\, C_{j}^{\alpha_i}  +\beta_i \, C_{j}^{\beta_i} ) + \alpha_{k+1}+\beta_{k+1}+\gamma_{k+1} \text{ whenever $j=j_k$ for some $k$; and} \\
\sum\limits_{i=2}^{k+1} (\alpha_i\, C_{j}^{\alpha_i}  +\beta_i \, C_{j}^{\beta_i} ) + \sigma_j   \text{ whenever $j_k<j<j_{k+1}$ and $j$ is not a gap}.
\end{cases}\]
\medskip
In all cases, we have $ C_{j}^{\alpha_i}=  C_{j-j_{i-1}}^{\alpha_1}$ and $C_{j}^{\beta_i}=C_{j-j_{i-1}}^{\beta_1}$; pushing the same argument further shows that each gap produces exactly two algebraically independent conditions.  
Indeed, the key fact is that the expansions of $C_{\ell_i}^{\alpha_1}$ and $C_{\ell_i}^{\beta_1}$ include linear summands $(a_{\ell_i}-c_{\ell_i})$ and $(b_{\ell_i}-c_{\ell_i})$ respectively (see \ref{equind}), which are algebraically independent. The same linearity property implies that the Severi stratum $\mc{V}_{(ab,ac,bc)}$ is unirational. 
\end{proof}

\end{document}